\documentclass{amsart}
\usepackage{amssymb,amsmath,cite,geometry}

\numberwithin{equation}{section}

\newtheorem{theorem}{Theorem}[section]
\newtheorem{lemma}{Lemma}[section]

\newtheorem{prop}{Proposition}[section]
\newtheorem{remark}{Remark}[section]

\newtheorem{con}{Conclusion}[section]

\newcommand{\R}{\mathbb{R}}

\allowdisplaybreaks[4]

\begin{document}
\title[Blowup of solutions  to the thermal boundary layer problem]
{Blowup of solutions  to the thermal boundary layer problem in two-dimensional incompressible heat conducting flow}

\author{Y.-G. Wang}
\address{Ya-Guang Wang
\newline\indent
School of Mathematical Sciences, MOE-LSC and SHL-MAC, Shanghai Jiao Tong University,
Shanghai 200240, P. R. China}
\email{ygwang@sjtu.edu.cn}

\author{S.-Y. Zhu}
\address{Shi-Yong Zhu
\newline\indent
School of Mathematical Sciences, Shanghai Jiao Tong University,
Shanghai 200240, P. R. China}
\email{shiyong\_zhu@sjtu.edu.cn}


\subjclass[2000]{35Q30, 76N20}

\date{}

\keywords{Thermal boundary layer problem, blowup of solutions}

\begin{abstract} In this paper, we study the formation of finite time singularities for the solution of the boundary layer equations in the two-dimensional incompressible heat conducting flow. We obtain that the first spacial derivative of the solution blows up in a finite time for the thermal boundary layer problem, for a kind of data which are analytic in the tangential variable but do not satisfy the Oleinik monotonicity condition, by constructing a Lyapunov functional. Moreover, it is observed that the buoyancy coming from the temperature difference in the flow may destabilize the thermal boundary layer. 

\end{abstract}

\maketitle


\section{Introduction}\setcounter{section}{1}\setcounter{equation}{0}
This paper is devoted to the study of blowup of solutions to the boundary layer equations in the two-dimensional incompressible heat conducting flow. Consider the following problem in $\{t>0,(x,y)\in\mathbb{R}\times\mathbb{R}^{+}\}$, 
\begin{equation}\label{1.1}
\begin{cases}
\partial_{x}u+\partial_{y}v=0,\\
\partial_{t}u+u\partial_{x}u+v\partial_{y}u=\partial_{y}^{2}u-\partial_{x}P-(\theta-\theta_{\infty}),\\
\partial_{t}\theta+u\partial_{x}\theta+v\partial_{y}\theta=\partial_{y}^{2}\theta+(\partial_{y}u)^{2},\\
u|_{y=0}=v|_{y=0}=0,~\theta|_{y=0}=\bar{\theta},~\lim\limits_{y\to+\infty}(u,\theta)=(U^{E},\theta^{E}),\\
u|_{t=0}=u_{0},~\theta|_{t=0}=\theta_{0},
\end{cases}
\end{equation}
where $(u,v)$ is the velocity field, $\theta$ is the temperature, $U^{E}(t,x)$, $\theta^{E}(t,x)$ and $P(t,x)$ are the traces at the boundary of the tangential velocity, temperature and pressure of the outer Euler flow respectively, $\theta_{\infty}$ is the reference temperature, which is assumed to be a positive constant. The states $U^{E},~\theta^{E}$ and $P$ are interrelated through
\begin{align}\label{1.2}
\partial_{t}U^{E}+U^{E}\partial_{x}U^{E}=-\partial_{x}P-\left(\theta^{E}-\theta_{\infty}\right),
\end{align}
and
\begin{align}\label{1.2-1}
\partial_{t}\theta^{E}+U^{E}\partial_{x}\theta^{E}=0.
\end{align}

The boundary layer problem \eqref{1.1} describes the behavior of the thermal layer and viscous layer,  in the small viscosity and heat conductivity limit for the following two-dimensional incompressible heat conducting Navier-Stokes system including the buoyancy force with non-slip boundary condition,
\begin{equation}\label{1.3}
\begin{cases}
\partial_{x}u+\partial_{Y}v=0,\quad\quad\quad\quad\quad\quad t>0,(x,Y)\in\mathbb{R}\times\mathbb{R}_{+},\\
\partial_{t}u+u\partial_{x}u+v\partial_{Y}u+\frac{1}{\rho}\partial_{x}p+ \nu g_{x}(\theta-\theta_{\infty})
=\epsilon\frac{1}{\rho}\left(\partial_{x}^{2}u+\partial_{Y}^{2}u\right),\\
\partial_{t}v+u\partial_{x}v+v\partial_{Y}v+\frac{1}{\rho}\partial_{Y}p+ \nu g_{Y}(\theta-\theta_{\infty})
=\epsilon\frac{1}{\rho}\left(\partial_{x}^{2}v+\partial_{Y}^{2}v\right),\\
\rho c\left(\partial_{t}\theta+u\partial_{x}\theta+v\partial_{Y}\theta\right)
=\kappa\left(\partial_{x}^{2}\theta+\partial_{Y}^{2}\theta\right)+\epsilon\Phi(t,x,Y),\\
u|_{Y=0}=v|_{Y=0}=0,~\theta|_{Y=0}=\bar{\theta},\\
u|_{t=0}=u_{in},~\theta|_{t=0}=\theta_{in}
\end{cases}
\end{equation}
where $\rho$ stands for density, $p$ is the pressure, $c$ is specific heat capacity, $\epsilon$ is the viscosity, $\kappa$ is the heat conductivity, the dissipation function $\Phi(t,x,Y)$ is given by
\begin{align*}
\Phi(t,x,Y)=2\left[(\partial_{x}u)^{2}+(\partial_{Y}v)^{2}\right]+\left(\partial_{Y}u+\partial_{x}v\right)^{2}.
\end{align*}
and $(\nu g_{x}(\theta-\theta_{\infty}),~\nu g_{Y}(\theta-\theta_{\infty}))$
comes from the buoyancy force, with $\nu$ being the coefficient of thermal expansion at temperature $\theta_{\infty}$, and $g=(g_{x},g_{Y})$ the vector of gravitational acceleration. One can consult \cite{S} for the physical background of this model. For simplicity of presentation, we assume that $\rho=c=\nu=1$ in \eqref{1.3}. In the case that the heat conductivity $\kappa$ is the same order as the viscosity $\epsilon$, e.g. $\kappa=\epsilon$, the thermal layer has the same thickness, $\sqrt{\epsilon}$, as the Prandtl viscous layer (\cite{Prandtl} ). In this case, as $g_{x}=1$ in \eqref{1.3} one can deduce the problem \eqref{1.1} for the boundary layer profiles by using multi-scale analysis with $y=\frac{Y}{\sqrt{\epsilon}}$. The derivation of \eqref{1.1} from \eqref{1.3} can be found in \cite{WZ-0}.

In the study of small viscosity limit for viscous flow in a domain with non-slip boundary condition,  in 1904 Prandtl\cite{Prandtl} introduced the boundary layer theory, and derived that the boundary layer profiles satisfy a degenerate parabolic equation for the tangential velocity coupled with the divergence-free constraint, which is called the Prandtl equations. Since then, there have been many interesting results on the well-posedness of the Prandtl equations.  The first local-in-time well-posedness result of the 2-D Prandtl equations was obtained by Oleinik and her collaborators (\cite{Oleinik,Ole}), by using the Crocco transformation under the assumption that the tangential velocity is strictly increasing with respect to the normal variable of the boundary. Recently, similar results were also obtained in the Sobolev spaces in \cite{AWXY} and \cite{MW} by using the energy method. In addition to the monotonic class, there are some well-posedness results in either analytic spaces or Gevrey spaces, cf. \cite{LCS,SC1, SC2, ZZ,CWZ,GM,LY} and references therein. On the other hand, without the monotonicity condition there are also some interesting results on instability and blowup of solutions to the Prandtl equations. When the initial tangential velocity has a nondegenerate critical point, certain ill-posedness results were obtained in \cite{GD, GN, LiuY} for the 2-D Prandtl equations in the Sobolev framework. For non-monotonic initial data, E and Engquist (\cite{EE}) showed that smooth solutions to the 2-D Prandtl equations shall blowup in a finite time. Moreover, for a kind of special outer flow and non-monotonic initial data, Kukavica, Vicol and Wang (\cite{KVW}) showed that the analytic solution to the 2-D Prandtl equations blows up in a finite time as well. 

As many models in application have heat conduction, it is important to study the small viscosity and heat conductivity limit in viscous flow with heat conduction. In \cite{WZ}, the authors  obtained a local-in-time well-posedness result in analytic class for the problem \eqref{1.1} by using the Littlewood-Paley theory. In this paper, motivated by \cite{EE,KVW}, for a kind of data which are analytic in the tangential variable but do not satisfy the Oleinik monotonicity condition, we shall show that the analytic solution of \eqref{1.1} blows up in a finite time in the Sobolev space. Moreover, the blowup result shows that the buoyancy coming from the temperature difference in the flow may destabilize the thermal boundary layer.

The remainder of this paper is arranged as follows.
In Section 2, we present some assumptions on underlying Euler flow, initial and boundary datum, then state the main result of this paper. In Section 3, we construct a Lyapunov functional. The blowup of the Lyapunov functional will be
proved in Section 4.

\section{Statement of the main result}\setcounter{section}{2}\setcounter{equation}{0}

Assume that the outer Euler flow for the problem \eqref{1.1} satisfies the following conditions for all $0\le t\le T$, 
\begin{equation}\label{2.1}
U^{E}|_{x=0}=0,~
\partial_{x}\theta^{E}|_{x=0}\geq0,
\end{equation}
which implies 
\begin{equation}\label{2.1-1}
(\theta^{E}+\partial_{x}P)|_{x=0}=\theta_{\infty},~
\partial_{t}\partial_{x}P|_{x=0}=0,
\end{equation}
by using \eqref{1.2} and \eqref{1.2-1}. Moreover, assume that the initial data and boundary data given in \eqref{1.1} satisfy
\begin{equation}\label{2.2}
\begin{cases}
u_{0}|_{x=0}=0,~\partial_{x}u_{0}|_{x=0}\leq 0,~\partial_{x}\theta_{0}|_{x=0}\geq0, \\[2mm]
(\bar{\theta}+\partial_{x}P)|_{x=0}=\theta_{\infty},~\partial_{x}\bar{\theta}|_{x=0}\geq0, 
\end{cases}
\end{equation}
which implies
\begin{equation}\label{2.3}
(\theta_{0}+\partial_{x}P(0,x))|_{x=0}=\theta_{\infty}
\end{equation}
from the second equation given in \eqref{1.1} for the smooth solution.

The main result of this paper is
\begin{theorem}
Assume that the outer Euler state, the initial and boundary data of the problem \eqref{1.1} satisfy the assumption given in Theorem 2.1 in \cite{WZ}, such that the problem \eqref{1.1} has a unique local solution $(u,v, \theta)$ analytic in the $x-$variable,  and the initial data $(u_{0},\theta_{0})$ belongs to the space $H^{2}(\mathbb{R}^{+})$ with respect to the $y-$variable and satisfies the compatibility conditions of \eqref{1.1} up to order one. Then, under the conditions \eqref{2.1}-\eqref{2.3}, for the solution $(u,v,\theta)$ of the problem \eqref{1.1}, $\partial_xu$ blows up in a finite time when the initial data $(u_{0},\theta_{0})$ is properly large.
\end{theorem}

\begin{remark}
Comparing with the blowup results obtained in \cite{EE} and \cite{KVW} for the classical Prandtl equations, we note that for the thermal boundary layer problem \eqref{1.1}, even when the pressure is favourable in the Prandtl sense (\cite{Prandtl, Ole, XZ}) which avoids the separation of the classical Prandtl layer, the thermal boundary layer profiles described by \eqref{1.1} may blow up in a finite time, which shows that the buoyancy in \eqref{1.1} may destabilize the thermal boundary layer. 
\end{remark}

We shall develop an idea similar to that given in \cite{KVW} to prove Theorem 2.1, by a contradiction argument. Assume that the problem \eqref{1.1} admits a solution $(u,v,\theta)$ on $[0,T]$, and there exists a constant $C_{T}$ depending only on $T$, such that
\begin{equation}\label{2.4}
\|\partial_{x}u(t,0,y)\|_{L^{\infty}([0,T]\times\mathbb{R}^{+})}
+\|\partial_{y}u(t,0,y)\|_{L^{\infty}([0,T]\times\mathbb{R}^{+})}
+\|\partial_{x}\theta(t,0,y)\|_{L^{\infty}([0,T]\times\mathbb{R}^{+})}\leq C_{T}.
\end{equation}
We shall construct a proper Lyapunov functional, and derive through a series estimates that this Lyapunov functional satisfies a differential inequality, from which one can obtain that $\|\partial_{x}u\|_{L^{\infty}([0,t]\times\mathbb{R}_{+}^{2})}$ blows up within the time interval $(0,T)$ by choosing initial data properly large.

In the following calculation, we shall use the notation $C$ to denote a generic positive constant independent of $T$, the initial boundary data and the underlying Euler flow, which may change from line to line.

\section{Construction of a Lyapunov functional}

Motivated by \cite{KVW}, in this section we construct a Lyapunov functional. The process of the construction can be divided into two steps. In the first step, we restrict the problem \eqref{1.1} on the $y-$axis by using the assumptions \eqref{2.1}-\eqref{2.3}, and transform the original problem into a one-space variable problem. In the second step, we construct a Lyapunov functional with a suitable weight function.

\subsection{Restriction of the thermal boundary layer problem on the $y-$axis}

Denote by $\tilde{\theta}(t,x,y)=\theta(t,x,y)-\theta_{\infty}+\partial_{x}P(t,0)$, and
\begin{equation*}
w(t,y)=u(t,0,y),~s(t,y)=\tilde{\theta}(t,0,y).
\end{equation*}

With the assumptions given \eqref{2.1}-\eqref{2.3}, by restricting the thermal boundary layer problem \eqref{1.1} on the $y-$axis, we deduce that $(w(t,y), s(t,y))$ satisfies the following problem in $\{t>0, y\in \R^+\}$,
\begin{equation}\label{3.1}
\begin{cases}
\partial_{t}w+w\partial_{x}u(t,0,y)+v(t,0,y)\partial_{y}w=\partial_{y}^{2}w-s,\\
\partial_{t}s+w\partial_{x}\theta(t,0,y)+v(t,0,y)\partial_{y}s=\partial_{y}^{2}s+(\partial_{y}w)^{2},\\
w|_{y=0}=0,~s|_{y=0}=0,~\lim\limits_{y\to+\infty}(w,s)=(0,0)\\
w|_{t=0}=0,~s|_{t=0}=0.
\end{cases}
\end{equation}

For the problem \eqref{3.1}, one can prove that
\begin{prop}\label{pro1}
Under the assumptions \eqref{2.1}-\eqref{2.4}, the problem \eqref{3.1} admits a unique trivial solution $w(t,y)=s(t,y)=0$ on [0,T].
\end{prop}
\begin{proof}
This proposition can be proved by using the energy method. Multiplying the first and second equations given in \eqref{3.1} by $w(t,y)$ and $s(t,y)$ respectively,  and integrating the resulting equations over $\mathbb{R}^{+}$ with respect to $y$, it follows that
\begin{equation}\label{3.2}
\begin{array}{lll}
\frac{1}{2}\frac{d}{dt}  \hspace{-.15in}& (& \hspace{-.15in}\|w(t)\|_{L^{2}(\mathbb{R}^{+})}^{2}
+ \|s(t)\|_{L^{2}(\mathbb{R}^{+})}^{2})
=\int_{0}^{\infty}w\partial_{y}^{2}wdy
-\int_{0}^{\infty}swdy+\int_{0}^{\infty}s\partial_{y}^{2}sdy+\int_{0}^{\infty}s(\partial_{y}w)^{2}dy\\[2mm]
& & \hspace{-.1in} -\int_{0}^{\infty}\partial_{x}u(t,0,y)w^{2}dy-\int_{0}^{\infty}v(t,0,y)w\partial_{y}wdy -\int_{0}^{\infty}\partial_{x}\theta(t,0,y)swdy-\int_{0}^{\infty}v(t,0,y)s\partial_{y}sdy\\[2mm]
& \le & \hspace{-.1in} (1+ \|\partial_{x}\theta(t,0,y)\|_{L^{\infty}([0,T]\times\mathbb{R}^{+})})(\|w(t)\|_{L^{2}(\mathbb{R}^{+})}^{2}
+ \|s(t)\|_{L^{2}(\mathbb{R}^{+})}^{2})-\int_{0}^{\infty}((\partial_{y}w)^{2}+(\partial_{y}s)^{2})dy\\[2mm]

& & \hspace{-.1in} + \|\partial_{x} u(t,0,y)\|_{L^{\infty}([0,T]\times\mathbb{R}^{+})}(\frac{3}{2}\|w(t)\|_{L^{2}(\mathbb{R}^{+})}^{2}
+ \frac{1}{2}\|s(t)\|_{L^{2}(\mathbb{R}^{+})}^{2})+\int_{0}^{\infty}s(\partial_{y}w)^{2}dy.
\end{array}
\end{equation}
by using integration by parts, the boundary conditions given in \eqref{3.1} and the divergence-free constraint given in \eqref{1.1}.

On the other hand, we have
\begin{align*}
\int_{0}^{\infty}s(\partial_{y}w)^{2}dy
\leq& \|\partial_{y}u(t,0,y)\|_{L^{\infty}([0,T]\times\mathbb{R}^{+})}\int_{0}^{\infty}|s\partial_{y}w|dy\\
\leq& \int_{0}^{\infty}(\partial_{y}w)^{2}dy
+\|\partial_{y}u(t,0,y)\|_{L^{\infty}([0,T]\times\mathbb{R}^{+})}^{2}\|s(t)\|_{L^{2}(\mathbb{R}^{+})}^{2}.
\end{align*}

Thus, from \eqref{3.2} we get that there is a constant $C>0$ such that 
\begin{equation}\label{3.2-1}
\frac{1}{2}\frac{d}{dt} (\|w(t)\|_{L^{2}(\mathbb{R}^{+})}^{2}
+ \|s(t)\|_{L^{2}(\mathbb{R}^{+})}^{2})
\le  C\left(C_{T}+C_{T}^{2}+1\right)\left(\|w\|_{L^{2}(\mathbb{R}^{+})}^{2}+\|s\|_{L^{2}(\mathbb{R}^{+})}^{2}\right),
\end{equation}
for the constant $C_T$ given in \eqref{2.4}.

By using the Gronwall inequality in \eqref{3.2-1}, it follows immediately,
\begin{equation*}
\|w(t)\|_{L^{2}(\mathbb{R}^{+})}^{2}+\|s(t)\|_{L^{2}(\mathbb{R}^{+})}^{2}=0,
\end{equation*}
for any $t\in[0,T]$.
\end{proof}

From Proposition 3.1, we know that $u(t,x,y)$ and $\tilde{\theta}(t,x,y)$ vanish at $x=0$, so we can write them in the following form
\begin{equation}\label{3.3}
u(t,x,y)=-x\bar{w}(t,x,y),~\tilde{\theta}(t,x,y)=-x\bar{s}(t,x,y).
\end{equation}
Moreover, from the assumption $U^{E}|_{x=0}=0$, one has
\begin{equation}\label{3.4}
U^{E}(t,x)=-x\bar{U}^{E}(t,x).
\end{equation}

Substituting \eqref{3.3} and \eqref{3.4} into $\eqref{1.1}_{2}$ and $\eqref{1.1}_{3}$, we get
\begin{equation}\label{3.5}
\begin{cases}
-x\partial_{t}\bar{w}+x\bar{w}(\bar{w}+x\partial_{x}\bar{w})-v(x\partial_{y}\bar{w})
=-x\partial_{y}^{2}\bar{w}+x\bar{P}(t,x)+x\bar{s},\\
-x\partial_{t}\bar{s}+x\bar{w}(\bar{s}+x\partial_{x}\bar{s})-v(x\partial_{y}\bar{s})
=-x\partial_{y}^{2}\bar{s}+(x\partial_{y}\bar{w})^{2},
\end{cases}
\end{equation}
with
$$\bar{P}(t,x)=-\frac{\partial_{x}P(t,x)-\partial_{x}P(t,0)}{x}.$$

Let $\tilde{w}(t,y)=\bar{w}(t,0,y)$ and $ \tilde{s}(t,y)=\bar{s}(t,0,y)$. Since $\{u, v,\theta\}$ is a solution of \eqref{1.1} which is analytic with respect to $x$, from \eqref{3.5} and \eqref{1.1} we deduce  that $(\tilde{w}, \tilde{s})$ satisfies the following problem
\begin{equation}\label{3.6}
\begin{cases}
\partial_{t}\tilde{w}-\tilde{w}^{2}+v(t,0,y)\partial_{y}\tilde{w}=\partial_{y}^{2}\tilde{w}-\bar{P}(t,0)-\tilde{s},\\
\partial_{t}\tilde{s}-\tilde{w}\tilde{s}+v(t,0,y)\partial_{y}\tilde{s}=\partial_{y}^{2}\tilde{s},\\
\tilde{w}|_{y=0}=0,~
\tilde{s}|_{y=0}=-\partial_{x}\bar{\theta}(t,0),~\lim\limits_{y\to+\infty}(\tilde{w}, \tilde{s})=(\bar{U}^{E}(t,0),-\partial_{x}\theta^{E}(t,0)),\\
\tilde{w}|_{t=0}=\tilde{w}_{0},~\tilde{s}|_{t=0}=\tilde{s}_{0}.
\end{cases}
\end{equation}

To get the conclusion of Theorem 2.1, it is sufficient to prove,

\begin{theorem}
Under the assumption \eqref{2.4},  the smooth solution to the problem \eqref{3.6} blows up within the time interval $(0,T)$ for a large class of initial data $(\tilde{w}_{0},\tilde{s}_{0})$.
\end{theorem}

\subsection{A lift to the target function}

As the solution $\tilde{w}$ of \eqref{3.6} may be negative, in order to construct a Lyapunov functional,  as in  \cite{KVW}, we add a lift function to $\tilde{w}$ by using the solution of the initial-boundary value problem of a nonhomogenous heat equation.

Let 
\begin{equation}\label{3.7-0}
C_{E}=-\!\!\inf\limits_{t\in[0,T]}\!\!\min\{\bar{U}^{E}(t,0),0\}, \qquad 
C_{P}=\!\!\sup\limits_{t\in[0,T]}\!\!\max\{\bar{P}(t,0),0\}.
\end{equation}
Consider the following problem in $\{t>0, y>0\}$, 
\begin{equation}\label{3.7}
\begin{cases}
\partial_{t}\phi-\partial_{y}^{2}\phi=C_{P}\\
\phi|_{y=0}=0,~\lim\limits_{y\to+\infty}\phi=C_{E}+C_{P}t,\\
\phi|_{t=0}=C_{E}{\rm Erf}\left(\frac{y}{2}\right)
\end{cases}
\end{equation}
where ${\rm Erf}(z)=\frac{2}{\sqrt{\pi}}\int_{0}^{z}e^{-s^{2}}ds$. It is easy to know that the solution of \eqref{3.7} is given as
\begin{align}\label{3.8}
\phi(t,y)=C_{E}{\rm Erf}\left(\frac{y}{\sqrt{4(t+1)}}\right)
+C_{P}t\left[\frac{y^{2}}{2t}\left({\rm Erf}\left(\frac{y}{\sqrt{4t}}\right)-1\right)
+{\rm Erf}\left(\frac{y}{\sqrt{4t}}\right)+\frac{y}{\sqrt{\pi t}}e^{-\frac{y^{2}}{4t}}\right].
\end{align}
As shown in \cite{KVW}, we have

\begin{lemma}\label{lemma1}
Let $\phi$ be the solution to the problem \eqref{3.7}. Then, for all $(t, y)\in \R^+\times \R^+$,  we have: 
{\rm(i)} $\partial_{y}\phi\geq0$, {\rm (ii)} $\phi\geq0$, {\rm (iii)} $\phi\leq C_{E}+C_{P}t$, {\rm (iv)} $\partial_{y}^{2}\phi\leq0$.
\end{lemma}

\begin{proof}
(i) From \eqref{3.7} and \eqref{3.8}, it is easy to know that 
\begin{equation*}
\partial_{y}\phi|_{y=0}>0, \quad  \partial_{y}\phi|_{t=0}\geq 0, \quad 
\lim\limits_{y\to+\infty}\partial_{y}\phi=0,
\end{equation*}
and
$\partial_{y}\phi$ satisfies the heat equation, so by using the minimum principle we deduce $\partial_{y}\phi\geq0$
for all $t\ge 0, y\ge 0$.

(ii) This is a direct corollary of (i) and $\phi|_{y=0}=0$.

(iii) Since $\phi|_{y=0}=0,~\phi|_{y\to\infty}=C_{E}+C_{P}t$, by using (i) we know
\begin{equation*}
\phi(t,y)\leq C_{E}+C_{P}t, \quad \forall (t,y)\in \R_+\times\R_+.
\end{equation*}

(iv) From the problem \eqref{3.7} we know that
\begin{align*}
\partial_{y}^{2}\phi|_{y=0}=-C_{P}\leq 0,
\quad \partial_{y}^{2}\phi|_{t=0}\leq 0, \quad \lim\limits_{y\to+\infty}\partial_{y}^{2}\phi=0.
\end{align*}
Since $\partial_{y}^{2}\phi$ is a solution of the heat equation, 
by using the maximum principle, we deduce that $\partial_{y}^{2}\phi\leq0$.
\end{proof}

Now let $a(t,y)=\tilde{w}(t,y)+\phi(t,y)$. Using $\eqref{1.1}_{1}$, we have
\begin{equation}\label{3.9}
v(t,x,y)=\int_{0}^{y}(\bar{w}(t,x,z)+x\partial_{x}\bar{w}(t,x,z))dz
=\partial_{y}^{-1}(\bar{w}(t,x,y)+x\partial_{x}\bar{w}(t,x,y)).
\end{equation}
From \eqref{3.6}, we know that $a(t,y)$ satisfies the following problem, 
\begin{equation}\label{3.10}
\begin{cases}
\partial_{t}a-\partial_{y}^{2}a=a^{2}-\partial_{y}^{-1}a\partial_{y}a+L[a]+G,\\
a|_{y=0}=0,~\lim\limits_{y\to\infty}a=C_{P}t+C_{E}+\bar{U}^{E}(t,0),\\
a|_{t=0}=a_{0},
\end{cases}
\end{equation}
where
\begin{equation}\label{3.10-1}
L[a]=-2a\phi+\partial_{y}^{-1}\phi \partial_{y}a+\partial_{y}^{-1}a \partial_{y}\phi,
\quad 
G=\phi^{2}-\partial_{y}^{-1}\phi\partial_{y}\phi+C_{P}-\bar{P}(t,0)-\tilde{s}.
\end{equation}
and
\begin{equation*}
a_{0}=\tilde{w}_{0}+C_{E}{\rm Erf}\left(\frac{y}{2}\right).
\end{equation*}

We will show that $a(t,y)$ is non-negative. To this end, we first have the following result.
\begin{lemma}\label{lemma2}
Under the condition \eqref{2.3}, if $(\tilde{w},\tilde{s})$ is a solution to the problem \eqref{3.6}, then $\tilde{s}(t,y)\leq0$ on $[0,T]\times\mathbb{R}^{+}$.
\end{lemma}
\begin{proof}
This fact can be proved by the classical maximum principle. Let $\zeta(t,y)=\tilde{s}(t,y)e^{-C_{T}t}$.  From the second equation of \eqref{3.6} we know that $\zeta(t,y)$ satisfies
\begin{align*}
\partial_{t}\zeta+(C_{T}-\tilde{w})\zeta+v(t,0,y)\partial_{y}\zeta=\partial^2_{y}\zeta.
\end{align*}
Using the maximum principle in the above equation, with \eqref{2.1}-\eqref{2.4} we know that $\zeta(t,y)\leq 0$, which gives $\tilde{s}(t,y)\leq 0$.
\end{proof}

Now, we can prove the following result on the non-negative property
of  $a(t,y)$.
\begin{lemma}\label{lemma3}
Assume that $a(t,y)$ is a classical solution to the problem \eqref{3.10} on $[0,T]$. If $a_{0}(y)>0$ for all $y>0$, then we have $a(t,y)\geq0$ on $[0,T]\times[0,\infty)$.
\end{lemma}
\begin{proof}
We argue by contradiction. Assume that  $a(t,y)$ reaches a negative value in $[0,T]\times(0,\infty)$, since $a_{0}(y)>0$ for all $y\in(0,\infty)$, there must exist a first time $t_{0}$ and an interior point $y_{0}>0$, such that
\begin{equation}\label{3.11-0}
a(t_{0},y_{0})=0,~\partial_{t}a(t_{0},y_{0})\leq 0,
~\partial_{y}a(t_{0},y_{0})=0,~\partial_{y}^{2}a(t_{0},y_{0})\geq 0,
\end{equation}
and $a(t_{0},y)\geq 0$ for all $y>0$. Thus we have
\begin{equation}\label{3.11-00}
\begin{cases}
(\partial_{y}^{2}a+a^{2}-\partial_{y}^{-1}a\partial_{y}a)(t_{0},y_{0})\geq 0,\\
L[a](t_{0},y_{0})=(\partial_{y}^{-1}a \partial_{y}\phi)(t_{0},y_{0})\geq 0.
\end{cases}
\end{equation}

On the other hand, from the definition of $G$ given in \eqref{3.10-1}, we know that 
\begin{align}\label{3.11}
G(t,y)\geq& \phi^{2}(t,y)-\partial_{y}^{-1}\phi(t,y)\partial_{y}\phi(t,y)
=\partial_{y}^{-1}(\phi\partial_{y}\phi-\partial_{y}^{-1}\phi\partial_{y}^{2}\phi)(t,y)\\
\geq& \partial_{y}^{-1}(\phi\partial_{y}\phi)(t,y)
=\frac{1}{2}\phi^{2}(t,y),\nonumber
\end{align}
by using Lemma 3.1, thus we get $G(t_{0},y_{0})>0$.
By using \eqref{3.11-00} and \eqref{3.11} in the first equation given in \eqref{3.10}, it follows that $\partial_{t}a(t_{0},y_{0})>0$, which is a contradiction to the assumption $\partial_{t}a(t_{0},y_{0})\le 0$ given in \eqref{3.11-0}.
\end{proof}

\subsection{Construction of a Lyapunov functional}

Now we construct a Lyapunov functional for the solution $a(t,y)$ of the problem \eqref{3.10}. Consider a function $\mathcal{Y}(y)$ satisfying 
\begin{align}\label{3.12}
\begin{cases}
\mathcal{Y}\geq 0,~\mathcal{Y}\in W^{2,\infty}(\mathbb{R}^{+}),~\mathcal{Y}\in L^{1}(\mathbb{R}^{+}),\\
\mathcal{Y}|_{y=0}=\mathcal{Y}|_{y=\infty}=0,~
\mathcal{Y}'|_{y=\beta}=\mathcal{Y}'|_{y=\infty}=0,\\
\mathcal{Y}''\leq 0,~\forall y\in(\alpha,\gamma),\\
\mathcal{Y}''\geq 0, ~\forall y\in [0,\alpha]\cup[\gamma,\infty),\\
\frac{1}{2}\mathcal{Y}'(0)+\mathcal{Y}'(\delta)+2\mathcal{Y}'(\gamma)\geq 0,\\
\inf_{y\in[\alpha,\delta]}\mathcal{Y}(y)(\delta-\alpha)^{-1}
+\mathcal{Y}'(\delta)+2\mathcal{Y}'(\gamma)\geq 0,\\
y\mathcal{Y}'\leq\mu\mathcal{Y},~\forall y\in(0,\beta),\\
\mathcal{Y}''\geq -\lambda\mathcal{Y}, ~\forall y\in(\alpha,\gamma),\\
\frac{{\mathcal{Y}'}^{2}}{\mathcal{Y}\mathcal{Y}''}\leq\sigma <1, ~\forall y\in[\delta,\infty),
\end{cases}
\end{align}
for fixed $0<\alpha<\beta<\gamma<\delta<\infty$, and $\mu,\lambda,\sigma>0$.

With such a weight function, we define a Lyapunov functional
\begin{align}\label{3.12-1}
\mathcal{G}(t)=\int_{0}^{\infty}a(t,y)\mathcal{Y}(y)dy,
\end{align}
for the solution $a(t,y)$ of the problem \eqref{3.10}.

\begin{remark}
It is not difficult to construct a function satisfying \eqref{3.12}. For example, one can take
\begin{equation*}
\mathcal{Y}(y)=
\begin{cases}
\frac{6}{5}Ay,~y\in[0,\frac{1}{2}), \\
-\frac{8A}{5}y^{3}+\frac{12A}{5}y^{2}+\frac{A}{5},~y\in[\frac{1}{2},1),\\
By^{3}+Dy^{2}+Ey+F,~y\in[1,2),\\
\frac{M^{n}}{(y+M-2)^{n}},~y\in[2,\infty),
\end{cases}
\end{equation*}
where $n$ and $M$ are positive constants and $$A=1+\frac{2n}{3M}+\frac{n(n+1)}{6M^{2}},~B=\frac{n}{3M}+\frac{n(n+1)}{3M^{2}},$$
$$D=-\frac{2n}{M}-\frac{3n(n+1)}{2M^{2}},~E=\frac{3n}{M}+\frac{2n(n+1)}{M^{2}},$$  $$F=1-\frac{2n}{3M}-\frac{2n(n+1)}{3M^{2}}.$$
In this case, $\alpha=\frac{1}{2}$, $\beta=1$, $\gamma=\frac{4nM+3n(n+1)}{2nM+2n(n+1)}$, $\delta=2$ and $\sigma=\frac{n}{n+1}$. One can choose a constant $n>1$ and then $M$ large enough to ensure that \eqref{3.12} holds, and $\mu$ and $\sigma$ can be obtained by calculation.
\end{remark}

\section{Blowup of the Lyapunov functional}

In this section, we will prove that the Lyapunov functional $\mathcal{G}(t)$ defined in \eqref{3.12-1} blows up within the time interval $(0,T)$ provided that the initial data is suitable large. Then from the definition of $\mathcal{G}(t)$ one can deduce that there exists $0<t_{*}< T$, such that $a(t,y)$ blows up at $t_{*}$, which directly concludes the blowup result claimed in Theorem 2.1.

For this, we have
\begin{prop}\label{pro2}
For the functional $\mathcal{G}(t)$ defined in \eqref{3.12-1}, there is $1<\eta<2$ such that
\begin{equation}\label{4.0}
\mathcal{G}'(t)
\geq 2(1-\eta^{-1})C_{\mathcal{Y}}^{-1}(\mathcal{G}(t))^{2}-[\lambda+(3+\mu)(C_{E}+C_{P}t)]\mathcal{G}(t)
\end{equation}
where $C_{\mathcal{Y}}=\|\mathcal{Y}\|_{L^{1}(\mathbb{R}^{+})}$,  the constants $C_E, C_P$ are given in \eqref{3.7-0}, and parameters $\lambda, \mu$ are given in \eqref{3.12}.
\end{prop}

If this proposition is true, then we have
\begin{align*}
\mathcal{G}(t)
\geq& \left(\frac{1}{\mathcal{G}(0)}-\frac{2-2\eta^{-1}}{C_{\mathcal{Y}}}\int_{0}^{t}\Psi(s)ds\right)^{-1}
\Psi(t),
\end{align*}
where
\begin{align*}
\Psi(t)=
e^{-\left[\lambda t+(3+\mu)C_{E}t+\frac{3+\mu}{2}C_{P}t^{2}\right]}.
\end{align*}
Therefore,  if the initial data  $\mathcal{G}(0)$ is large enough, e.g.
\begin{align*}
\mathcal{G}(0)
\geq & \frac{C_{\mathcal{Y}}}{2-2\eta^{-1}}\left(\int_{0}^{\frac{T}{2}}e^{-\left[\lambda t+(3+\mu)C_{E}(T)t+\frac{3+\mu}{2}C_{P}(T)t^{2}\right]}dt\right)^{-1}, 
\end{align*}
then
$\mathcal{G}(t)$ blows up within $(0,T)$. Thus, the blowup result claimed in Theorem 2.1 follows.

\vspace{.1in}
\underline{\bf Proof of Proposition \ref{pro2}:}
By using the first equation given in  \eqref{3.10}, and the non-negative property of $G$ obtained in \eqref{3.11}, we have
\begin{align}\label{4.1}
\frac{d}{dt}\mathcal{G}
=&\int_{0}^{\infty}\partial_{y}^{2}a\mathcal{Y}dy
+\int_{0}^{\infty}\left(a^{2}-(\partial_{y}^{-1}a)\partial_{y}a\right)\mathcal{Y}dy
+\int_{0}^{\infty}L[a]\mathcal{Y} dy
+\int_{0}^{\infty}G\mathcal{Y} dy\\
\geq& \int_{0}^{\infty}\partial_{y}^{2}a\mathcal{Y}dy
+\int_{0}^{\infty}\left(a^{2}-(\partial_{y}^{-1}a)\partial_{y}a\right)\mathcal{Y}dy
+\int_{0}^{\infty}L[a]\mathcal{Y} dy\nonumber\\
= & \int_{0}^{\infty}\partial_{y}^{2}a\mathcal{Y}dy
+\frac{1}{2}\int_{0}^{\infty}\partial_{y}^{2}[(\partial_{y}^{-1}a)^{2}]\mathcal{Y}dy
-2\int_{0}^{\infty}(\partial_{y}^{-1}a)\partial_{y}a\mathcal{Y}dy
+\int_{0}^{\infty}L[a]\mathcal{Y} dy\nonumber,
\end{align}
by noting 
$$\partial_y^2((\partial_y^{-1}a)^2)=2(a^2+(\partial_y^{-1}a)(\partial_ya)).$$

By  integration by parts and using \eqref{3.12}, we obtain
\begin{align}\label{4.2}
-2\int_{0}^{\infty}\partial_{y}^{-1}a\partial_{y}a\mathcal{Y}dy
=&2\int_{0}^{\infty}a^{2}\mathcal{Y}dy
+2\int_{0}^{\infty}(\partial_{y}^{-1}a)a\mathcal{Y}'dy\\
=&2\int_{0}^{\infty}a^{2}\mathcal{Y}dy
-\frac{1}{2}\int_{0}^{\beta}(\partial_{y}^{-1}a)^{2}\mathcal{Y}''dy
+\int_{0}^{\beta}(\partial_{y}^{-1}a)a\mathcal{Y}'dy\nonumber\\
&+2\int_{\beta}^{\infty}(\partial_{y}^{-1}a)a\mathcal{Y}'dy\nonumber.
\end{align}

Plugging  \eqref{4.2} into \eqref{4.1}, and by integration by parts it follows that
\begin{align}\label{4.3}
\frac{d}{dt}\mathcal{G}
\geq& \int_{0}^{\infty}a\mathcal{Y}''dy
+\frac{1}{2}\int_{\beta}^{\infty}[(\partial_{y}^{-1}a)^{2}]\mathcal{Y}''dy
+2\int_{0}^{\infty}a^{2}\mathcal{Y}dy
+\int_{0}^{\beta}(\partial_{y}^{-1}a)a\mathcal{Y}'dy\\
&+2\int_{\beta}^{\infty}(\partial_{y}^{-1}a)a\mathcal{Y}'dy
+\int_{0}^{\infty}L[a]\mathcal{Y} dy\nonumber\\
:=&\sum_{i=1}^{6}\mathcal{R}_{i}\nonumber,
\end{align}
with obvious notions $\mathcal{R}_{i}(1\leq i\leq 6)$. Now we study each $\mathcal{R}_{i}$ respectively.

i) By using \eqref{3.12}, we can bound $\mathcal{R}_{1}$ as follows,
\begin{align}\label{4.4}
\mathcal{R}_{1}
\geq -\lambda\int_{0}^{\infty}a\mathcal{Y}dy
= -\lambda\mathcal{G}.
\end{align}

ii) Decompose $\mathcal{R}_{2}$ into two parts,
\begin{align*}
\mathcal{R}_{2}
=\frac{1}{2}\int_{\beta}^{\delta}[(\partial_{y}^{-1}a)^{2}]\mathcal{Y}''dy
+\frac{1}{2}\int_{\delta}^{\infty}[(\partial_{y}^{-1}a)^{2}]\mathcal{Y}''dy.
\end{align*}
By using \eqref{3.12} and $\mathcal{Y}'(\delta)\le 0$, we have
\begin{align*}
\frac{1}{2}\int_{\beta}^{\delta}[(\partial_{y}^{-1}a)^{2}]\mathcal{Y}''dy
\geq& \frac{1}{2}\int_{\beta}^{\gamma}\left((\partial_{y}^{-1}a)(t,\gamma)\right)^{2}\mathcal{Y}''dy
+\frac{1}{2}\int_{\gamma}^{\delta}\left((\partial_{y}^{-1}a)(t,\gamma)\right)^{2}\mathcal{Y}''dy\\
=&\frac{1}{2}\mathcal{Y}'(\delta)\left((\partial_{y}^{-1}a)(t,\gamma)\right)^{2}
\geq \mathcal{Y}'(\delta)\left((\partial_{y}^{-1}a)(t,\alpha)\right)^{2}
+\mathcal{Y}'(\delta)\left(\int_{\alpha}^{\gamma}ady\right)^{2} \nonumber,
\end{align*}
which implies
\begin{align}\label{4.5}
\mathcal{R}_{2}
\geq \mathcal{Y}'(\delta)\left((\partial_{y}^{-1}a)(t,\alpha)\right)^{2}
+\mathcal{Y}'(\delta)\left(\int_{\alpha}^{\gamma}ady\right)^{2}
+\frac{1}{2}\int_{\delta}^{\infty}[(\partial_{y}^{-1}a)^{2}]\mathcal{Y}''dy.
\end{align}

iii) Since by using the H\"{o}lder inequality,
\begin{align*}
\int_{\alpha}^{\delta}a^{2}\mathcal{Y}dy
\geq \inf_{y\in[\alpha,\delta]}\mathcal{Y}(y)(\delta-\alpha)^{-1}\left(\int_{\alpha}^{\delta}ady\right)^{2},
\end{align*}
we get that
\begin{align}\label{4.6}
\mathcal{R}_{3}
\geq 2\left(\int_{0}^{\alpha}+\int_{\delta}^{\infty}\right)a^{2}\mathcal{Y}dy
+\int_{\alpha}^{\delta}a^{2}\mathcal{Y}dy
+\inf_{y\in[\alpha,\delta]}\mathcal{Y}(y)(\delta-\alpha)^{-1}\left(\int_{\alpha}^{\delta}ady\right)^{2}.
\end{align}

iv) By using the property of the weight function given in  \eqref{3.12}, it is easy to deduce that
\begin{align}\label{4.7}
\mathcal{R}_{4}
\geq \mathcal{Y}'(0)\int_{0}^{\alpha}a\partial_{y}^{-1}a dy
\geq \frac{1}{2}\mathcal{Y}'(0)\left((\partial_{y}^{-1}a)(t, \alpha)\right)^{2}.
\end{align}

v) Decompose $\mathcal{R}_{5}$ into two parts,
\begin{align*}
\mathcal{R}_{5}
=2\int_{\beta}^{\delta}(\partial_{y}^{-1}a)a\mathcal{Y}'dy
+2\int_{\delta}^{\infty}(\partial_{y}^{-1}a)a\mathcal{Y}'dy
:=\mathcal{R}_{5}^{1}+\mathcal{R}_{5}^{2}.
\end{align*}
For $\mathcal{R}_{5}^{1}$, by using \eqref{3.12} we have
\begin{align}\label{4.8}
\mathcal{R}_{5}^{1}
\geq 2\mathcal{Y}'(\gamma)\int_{\beta}^{\delta}(\partial_{y}^{-1}a)ady
= \mathcal{Y}'(\gamma)\left((\partial_{y}^{-1}a)(t,\delta)\right)^{2}
\geq 2\mathcal{Y}'(\gamma)\left(\left((\partial_{y}^{-1}a)(t,\alpha)\right)^{2}
+\left(\int_{\alpha}^{\delta}ady\right)^{2}\right),
\end{align}
by noting $\mathcal{Y}'(\gamma)\le 0$.
For $\mathcal{R}_{5}^{2}$, by using the Young inequality it follows that for any fixed 
$\eta>0$, 
\begin{align}\label{4.9}
\mathcal{R}_{5}^{2}
\geq & -\frac{\eta}{2}\int_{\delta}^{\infty}\left(\partial_{y}^{-1}a\right)^{2}\frac{{\mathcal{Y}'}^{2}}{\mathcal{Y}}dy
-2\eta^{-1}\int_{\delta}^{\infty}a^{2}\mathcal{Y}dy,\nonumber\\
\geq & -\frac{\eta\sigma}{2}\int_{\delta}^{\infty}\left(\partial_{y}^{-1}a\right)^{2}\mathcal{Y}''dy
-2\eta^{-1}\int_{\delta}^{\infty}a^{2}\mathcal{Y}dy,
\end{align}
by using \eqref{3.12}.

Thus, from \eqref{4.8} and \eqref{4.9} we deduce that
\begin{align}\label{4.10}
\mathcal{R}_{5}
\geq  2\mathcal{Y}'(\gamma)\left(\left((\partial_{y}^{-1}a)(t,\alpha)\right)^{2}
+\left(\int_{\alpha}^{\delta}ady\right)^{2}\right)
-\frac{\eta\sigma}{2}\int_{\delta}^{\infty}\left(\partial_{y}^{-1}a\right)^{2}\mathcal{Y}''dy
-2\eta^{-1}\int_{\delta}^{\infty}a^{2}\mathcal{Y}dy.
\end{align}

vi) By integration by parts and using Lemmas 3.1 and 3.3, we obtain
\begin{align*}
\mathcal{R}_{6}
=& -2\int_{0}^{\infty}a\phi\mathcal{Y}dy
+\int_{0}^{\infty}\partial_{y}^{-1}\phi \partial_{y}a\mathcal{Y}dy
+\int_{0}^{\infty}\partial_{y}^{-1}a \partial_{y}\phi\mathcal{Y}dy\\
=& -3\int_{0}^{\infty}a\phi\mathcal{Y}dy
-\int_{0}^{\infty}\partial_{y}^{-1}\phi a\mathcal{Y}'dy
+\int_{0}^{\infty}\partial_{y}^{-1}a \partial_{y}\phi\mathcal{Y}dy\\
\geq& -3\int_{0}^{\infty}a\phi\mathcal{Y}dy
-\int_{0}^{\beta}\partial_{y}^{-1}\phi a\mathcal{Y}'dy,
\end{align*}
by using $\mathcal{Y}'(y)\le 0$ for all $\beta<y<+\infty$.

Then, by using Lemma 3.1(iii), $\mathcal{Y}'(y)\ge 0$ on $[0, \beta]$ and \eqref{3.12} we deduce that
\begin{align}\label{4.11}
\mathcal{R}_{6}
\geq& -3(C_{E}+C_{P}t)\int_{0}^{\infty}a\mathcal{Y}dy
-(C_{E}+C_{P}t)\int_{0}^{\beta}ay\mathcal{Y}'dy \\
\geq& -(3+\mu)(C_{E}+C_{P}t)\mathcal{G}.\nonumber
\end{align}

Combining \eqref{4.4}-\eqref{4.11} with \eqref{4.3} we obtain
\begin{align}\label{4.12}
\frac{d}{dt}\mathcal{G}
\geq& -[\lambda+(3+\mu)(C_{E}+C_{P}t)]\mathcal{G}
+\left[\frac{1}{2}\mathcal{Y}'(0)+\mathcal{Y}'(\delta)+2\mathcal{Y}'(\gamma)\right]
\left(\partial_{y}^{-1}a(\alpha)\right)^{2}\\
&+\left[\inf_{y\in[\alpha,\delta]}\mathcal{Y}(y)(\delta-\alpha)^{-1}
+\mathcal{Y}'(\delta)+2\mathcal{Y}'(\gamma)\right]\left(\int_{\alpha}^{\delta}ady\right)^{2}\nonumber\\
&+\frac{1-\eta\sigma}{2}\int_{\delta}^{\infty}[(\partial_{y}^{-1}a)^{2}]\mathcal{Y}''dy
+2\int_{0}^{\alpha}a^{2}\mathcal{Y}dy
+\int_{\alpha}^{\delta}a^{2}\mathcal{Y}dy\nonumber\\
&+(2-2\eta^{-1})\int_{\delta}^{\infty}a^{2}\mathcal{Y}dy\nonumber.
\end{align}

Since $\sigma<1$ given in \eqref{3.12}, one can choose $1<\eta<2$ such that $\eta\sigma<1$, 
Thus, by using \eqref{3.12}, from \eqref{4.12} we get
\begin{align*}
\frac{d}{dt}\mathcal{G}
\geq& -[\lambda+(3+\mu)(C_{E}+C_{P}t)]\mathcal{G}
+(2-2\eta^{-1})\int_{0}^{\infty}a^{2}\mathcal{Y}dy\\
\geq& -[\lambda+(3+\mu)(C_{E}+C_{P}t)]\mathcal{G}
+(2-2\eta^{-1})C_{\mathcal{Y}}^{-1}\mathcal{G}^{2},
\end{align*}
where $C_{\mathcal{Y}}=\|\mathcal{Y}\|_{L^{1}(\mathbb{R}^{+})}$.

\vspace{.2in}

\noindent{\bf Acknowledgments}
This research was partially supported by
National Natural Science Foundation of China (NNSFC) under Grant No. 11631008.

\end{document}